\documentclass[10pt]{amsart}
\usepackage{a4wide}
\usepackage{amsfonts}
\usepackage{psfrag}

\usepackage{amsmath}
\usepackage{amssymb}
\usepackage{amsthm}
\usepackage{graphicx}
\usepackage{color}

\usepackage{hyperref}

\newtheorem{proposition}{Proposition}
\newtheorem {lemma}{Lemma}
\newtheorem {theorem}{Theorem}

\newtheorem {cor}{Corollary}

\theoremstyle{remark}
\newtheorem{remark}[theorem]{Remark}
\numberwithin{equation}{section}
\numberwithin{equation}{section}
\numberwithin{equation}{section}

\newcommand{\Pbb}[1]{\Pb\lb #1\rb}

\newcommand{\Ebb}[1]{\Eb\lbb #1\rbb}

\newcommand{\LL}{L\'{e}vy }

\newcommand{\LLPs}{L\'{e}vy processes }

\newcommand{\tto}[1]{_{#1\to 0}}

\newcommand{\ttinf}[1]{_{#1\to \infty}}

\newcommand{\Eb}{\mathbb{E}}

\newcommand{\Pb}{\mathbb{P}}

\newcommand{\lb}{\left (}
\newcommand{\rb}{\right )}
\newcommand{\lbb}{\left [}
\newcommand{\rbb}{\right ]}
\newcommand{\labs}{\left |}
\newcommand{\rabs}{\right |}
\newcommand{\lbrb}[1]{\lb #1 \rb}
\newcommand{\lbbrbb}[1]{\lbb#1\rbb}
\newcommand{\labsrabs}[1]{\labs#1\rabs}

\begin{document}
\title{On the Range of Subordinators}

\author{Mladen Savov}
   
\address{Department of Mathematics, University of Reading, Whiteknights campus, Reading RG6 6AX, United Kingdom}
\maketitle
\begin{abstract}
In this note we look into detail at the box-counting dimension of subordinators. Given that $X$ is a non-decreasing \LL process, which is not a compound Poisson process, we show that in the limit, the minimum number of boxes of size $\delta$ that cover the range of $\lbrb{X_s}_{s\leq t}$ is a.s.\,\,of order $t/U(\delta)$, where $U$ is the potential function of $X$. This is a more refined result than the lower and upper index of the box-counting dimension computed in \cite[Ch. 5, Th. 5.1]{B99} which deals with the asymptotic number of boxes at logarithmic scale.
\end{abstract}

\section{Introduction and Results}

	In this note we consider the minimal number of intervals that cover the range of a given subordinator $X:=(X_s)_{s\geq 0}$. The problem of studying the set properties of the range of \LLPs in general and subordinators in particular have a long history. Various measures for dimension have been discussed for the range of \LL processes. We refer to \cite{B99} and \cite{B96} for more information on the range of subordinators and to the work of \cite{KX05}, \cite{KX06}, for results on more general \LL processes. In this work we improve the results on the box-counting dimension of subordinators presented in \cite[Ch.5]{B99} by showing the a.s.\,\,convergence of the random variable that counts the minimal number of intervals that cover the range of a subordinator up to a given time $t$ rescaled by the potential function of the subordinator $X$. Previously the behaviour of the number of intervals has only been discussed at a logarithmic scale ( see Remark \ref{rem:R1} ), and even then a precise convergence has not been available. \newline
	
	Recall that for any subordinator $X$ defined on some probability space $\lbrb{\Omega,\mathcal{F},\Pb}$, we have that
	\begin{equation}\label{eq:Laplace}
	\Ebb{e^{-\lambda X_t}}=e^{-t\Phi(\lambda)},\,\,\,\,\text{ for all $\lambda>0$},
	\end{equation}
	where we call $\Phi\colon[0,\infty)\mapsto[0,\infty)$ the \LL\!\!-Khintchine exponent of the subordinator $X$. The function $\Phi$ has the representation
	\begin{equation}\label{eq:Phi}
	\Phi\lbrb{\lambda}=d\lambda+\int_{0}^{\infty}\lbrb{1-e^{-\lambda y}}\Pi(dy),
	\end{equation}
	where $d\geq 0$ is the linear drift of the subordinator $X$ and the measure $\Pi$, satisfying $\int_{0}^{\infty}\min\{1,x\}\Pi(dx)<\infty$, describes the intensity and the size of the jumps of $X$. In the sequel we shall assume either that the jumps of $X$ are \textbf{infinitely} many on any finite interval of time (\,that  is, $\Pi\lbrb{0,\infty}=\infty$\,) or, if the jumps of $X$ are finitely many on any finite interval of time  (\,that  is, $\Pi\lbrb{0,\infty}<\infty$\,),  that $d>0$. This is
	equivalent to $X$  not being a compound Poisson process.\newline 
	 
	Denote by $N(t,\delta)$ the minimal number of intervals of length at most $\delta$ that are needed to cover the range of $X$ up to time $t>0$. The most economic covering of this type is constructed in the following way: Denote by $T_{0}=0$ and 
	\begin{align*}
		T_{1}(\delta)&=\inf\{t\geq 0:\,\,X_{t}>\delta\},\\
		T_{n+1}(\delta)&=\inf\{t\geq T_{n}(\delta):\,\,X_{t}-X_{T_{n}(\delta)}>\delta\},
	\end{align*}
	which clearly implies that $\{N(t,\delta)\geq k\}=\{T_{k}(\delta)\leq t\}$. For clarity we write $T_k(\delta)=T_k$ when there is no ambiguity. The minimal covering of the range of $X$ up to any time $t>0$ is the collection of random intervals $\{\lbrb{X_{T_{n-1}},X_{T_{n-}}}\}_{\{n\geq 0,T_{n-1}<t\}}$. In the following, we write $\eta_{i}:=T_i-T_{i-1}$ and note that $(\eta_{i})_{i\geq 1}$ is a sequence of independent identically distributed random variables.
	We will frequently use the $q$-potentials of $X$ defined by
	\begin{align}\label{eq:Q-Potentials}
		&U_{q}(\delta)=\int_{0}^{\infty}e^{-qt}\Pbb{X_{t}\leq \delta}dt=\frac{1}{q}-\frac{1}{q}\Ebb{e^{-qT_1(\delta)}},\,\,\,\,\text{ for all $q>0$}
	\end{align}
	and we abbreviate $U(\delta)\colon\!\!=U_{0}(\delta)=\Ebb{T_1(\delta)}$ noting that the first identity in \eqref{eq:Q-Potentials} makes sense even when $q=0$. Our aim is to show that $U(\delta)N(t,\delta)$ converges to $t$ almost surely. A case where this would fail is the compound Poisson process case in which it is apparent that $U(0+)>0$.

	 \begin{theorem}\label{t1}
	 	If $X$ is a subordinator with $\Pi\lbrb{0,\infty}=\infty$ or $d>0$, if $\Pi\lbrb{0,\infty}<\infty$, then
		 \begin{equation}\label{eq2}
			 \lim_{\delta\to 0} U(\delta)N(t,\delta)=t
		 \end{equation} 
		 almost surely.
	 \end{theorem}
      \begin{remark}\label{rem:R1}
      We note that this result improves the result presented in \cite[Ch. 5, Th. 5.1]{B99} wherein 
      \begin{equation}\label{eq:lowerIndex}
      \liminf\tto{\delta}\frac{\ln\lbrb{N(t,\delta)}}{\ln\lbrb{\frac{1}{\delta}}}=\liminf\ttinf{\lambda}\frac{\ln\lbrb{\Phi(\lambda)}}{\ln\lbrb{\lambda}}=\colon\underline{\mathrm{ind}}(\Phi)
      \end{equation}
      and  \begin{equation}\label{eq:upperIndex}
      \limsup\tto{\delta}\frac{\ln\lbrb{N(t,\delta)}}{\ln\lbrb{\frac{1}{\delta}}}=\limsup\ttinf{\lambda}\frac{\ln\lbrb{\Phi(\lambda)}}{\ln\lbrb{\lambda}}=\colon\overline{\mathrm{ind}}(\Phi).
      \end{equation}
      The quantity $\liminf\tto{\delta}\frac{\ln\lbrb{N(t,\delta)}}{\ln\lbrb{\frac{1}{\delta}}}$ is known as the lower box-counting dimension whereas\\[-0.1cm] $\limsup\tto{\delta}\frac{\ln\lbrb{N(t,\delta)}}{\ln\lbrb{\frac{1}{\delta}}}$ is the upper box-counting dimension, see \cite[Chap. 5]{B99}.
      \begin{remark}\label{rem:R2}
       Note that the relation $U(\delta)\asymp \frac{1}{\Phi\lbrb{\frac{1}{\delta}}}$, i.e. $\frac{C_1}{\Phi\lbrb{\frac{1}{\delta}}}\leq U(\delta)\leq \frac{C_2}{\Phi\lbrb{\frac{1}{\delta}}}$ for two absolute constants $0<C_1<C_2<\infty$, see \cite[Ch III,Prop. 1]{B96}, which leads to $\frac{\labsrabs{\ln\lbrb{U(\delta)}}}{\ln\lbrb{\frac{1}{\delta}}}\sim \frac{\labsrabs{\ln{\Phi\lbrb{\frac{1}{\delta}}}}}{\ln\lbrb{\frac{1}{\delta}}}$ almost surely, shows in alternative way, due to Theorem \ref{t1}, that it may happen that $\underline{\mathrm{ind}}(\Phi)<\overline{\mathrm{ind}}(\Phi)$. Theorem \ref{t1}, however, further demonstrates that the scale $\ln\lbrb{\frac{1}{\delta}}$ is not the right one for $\ln\lbrb{N(t,\delta)}$ and shows the existence of a correct deterministic scale even for $N(t,\delta)$.
      \end{remark}
      \begin{remark}\label{rem:R3}
        The notion of lower and upper box dimension is tightly related to the notion of packing dimension and packing measure. In fact $\overline{\mathrm{ind}}(\Phi)=\mathrm{dim}_P(X),$ see e.g. \cite[Chap. 5, p.42]{B99}, where $\mathrm{dim}_P(X)$ is the packing dimension of the range of the subordinator $X$ for $t=1$. Moreover, the possible packing measures generated by the measure functions $\varphi$ have been extensively studied, see \cite{FT92} for more detail. The notable conclusion is that unless the subordinator is not of Cauchy type, see \cite[Section 4]{FT92}, then  the $\varphi$-packing measure is either zero or infinity. To our understanding, this does not allow, these otherwise very refined results, to shed light on the problem we discuss, i.e. the a.s.\,\,behaviour of $N(t,\delta)$.
      \end{remark}
      \end{remark}
      \section{Some applications}
      The first remark is the very precise a.s.\,behaviour we can obtain for $N(t,\delta),$ as $\delta\to 0$, whenever $d>0$.
      \begin{cor}\label{cor:d>0}
      Let $d>0$. We have that a.s., for any $t>0$, as $x\to 0$,
      \begin{align}\label{eq:AsympExpd>0}
       N(t,x)\sim \frac{t}{\sum_{n\geq 0}\frac{(-1)^n}{d^{n+1}}\int_{0}^{x}1*\bar{\Pi}^{*n}(y)dy},
      \end{align}
      where $f*g(x)=\int_{0}^{x}f(x-y)g(y)dy$ and $\bar{\Pi}(x)=\Pi\lbrb{x,\infty}$.
      \end{cor}
      The second remark is also immediate but we formulate it for convenience. Note that the information at logarithmic level, namely using the available results about $\ln\lbrb{N(t,\delta)}$, hides away a good deal of precision as to the fluctuations of $N(t,x)$ which are due to a second order variation in $\Phi\lbrb{\lambda}$. This is particularly apparent when $\alpha=0$ in the statement.
      \begin{cor}\label{cor:RV}
      Let $\Phi(\lambda)\sim \lambda^\alpha L(\lambda),$ as $\lambda\to\infty$, $\alpha\in[0,1]$ and $L$ a slowly varying function. Then 
      \begin{align}\label{eq:AsympExpRV}
             N(t,\delta)\stackrel{a.s.}{\sim} t\frac{\Gamma(1+\alpha)}{\delta^{\alpha}}L\lbrb{\frac{1}{\delta}},
            \end{align}
                  where $\Gamma(x)$ is the celebrated Gamma function.
      \end{cor}
      
      Next, we turn our attention to an interesting property that can be deduced using Theorem \ref{t1}. Put
      \begin{equation}\label{eq:HausdorffFunction}
      f(x)=\frac{\ln{|\ln x|}}{\Phi\lbrb{\frac{\ln{|\ln x|}}{x}}}.
      \end{equation}
      Then it is known from \cite{FP71} that with $I_.$ standing for a generic finite interval
      \begin{equation}\label{eq:measure}
      \lim\tto{\delta}\lbrb{\inf\big\{\sum_{i}f(I_i):\,\lbrb{X_s}_{s\leq 1}\subset \bigcup_i I_i;\,\max_{i}|I_i|<\delta \big\}}=1\quad\text{a.s.}
      \end{equation}
      The following result shows that usually the efficient covering of the set $\lbrb{X_s}_{s\leq 1}$ is not achieved by intervals of length proportionate to $\delta$, as $\delta\to 0$. For any finite collection of sets $\mathcal{A}_\delta=\{i:\,|I_i|<\delta\}$ denote by $\mathcal{A}^c_\delta=\{I\in\mathcal{A}_\delta:|I|>c\delta\}$, for any $c\in(0,1)$. Then we have the following result.
      \begin{proposition}\label{prop:1}
      Let $X$ be a subordinator satisfying 
      \begin{equation}\label{eq:notCauchy}
      \lim\ttinf{x}\frac{\Phi(x)\ln\ln(x)}{\Phi\lbrb{x\ln\ln(x)}}=\infty
      \end{equation}
      then if $\limsup\tto{\delta}\sum_{I_i\in \mathcal{A}_\delta}f(I_i)<\infty$, for some collection of coverings of $\lbrb{X_s}_{s\leq 1}$, that is,  $\lbrb{\mathcal{A}_\delta}_{\delta<1}$, we have that, for any $c\in(0,1)$,
      \begin{equation}\label{eq:result}
      \lim\tto{\delta}\frac{\labsrabs{\mathcal{A}^c_\delta}}{N(1,\delta)}=0\quad \text{a.s.}
      \end{equation}
      \end{proposition}
      \begin{remark}\label{R5}
      Note that using \cite[Chap. 3, Prop 1]{B96} one checks that $\Phi(x)f(x)\geq 1$. Also \eqref{eq:notCauchy} may fail to hold whenever $\Phi(x)\sim xL(x),$ as $x\to\infty$ with $L(x)$-\,some slowly varying function. However, if $\Phi(x)$ is not close to linear behaviour then \eqref{eq:notCauchy} does hold and one deduces from \eqref{eq:result} that the number of the intervals proportionate to $\delta$ in efficient coverings of $\lbrb{X_s}_{s\leq 1}$ is of a smaller magnitude than the most economic covering $N(1,\delta)$. 
      \end{remark}
	 \section{ Proof of Theorem \ref{t1}}

	 We start the proof by showing a couple of auxiliary results. 

	\begin{lemma}\label{L1}
		If $X$ is a subordinator with $\Pi\lbrb{0,\infty}=\infty$ or, if\,\,\,$\Pi\lbrb{0,\infty}<\infty$, then $d>0$, the following convergence holds:
		\begin{equation}\label{eq1}
			\lim_{\delta\to 0}U(\delta)N(t,\delta)=t\quad \text{ in probability}.
		\end{equation}
	\end{lemma}

	\begin{proof}
		First note that the definitions above directly imply that $\Pbb{N(t,\delta)\geq k}=\Pbb{\sum_{i=1}^{k}\eta_{i}\leq t}$, where we recall that $\eta_i=T_{i}-T_{i-1}$. Let us first study the sequence of random variables $M(\alpha,\delta):=\sum_{i=1}^{[\frac{\alpha} {U(\delta)}]}\eta_{i}$, where $[x]=\max\{n\geq 0:x\geq n\}$. We compute their Laplace transform, using the fact that  the random variables $\eta_i$ are independent and identically distributed, to get
		\begin{align*}
			\Ebb{e^{-\theta M(\alpha,\delta)}}=\lbrb{\Ebb{e^{-\theta \eta_{1}}}}^{[\frac{\alpha} {U(\delta)}]}=\lbrb{\Ebb{e^{-\theta T_{1}(\delta)}}}^{[\frac{\alpha} {U(\delta)}]}.
		\end{align*}
       Using \eqref{eq:Q-Potentials}, namely $\Ebb{e^{-\theta T_{1}(\delta)}}=1-\theta U_{\theta}(\delta)$, we obtain that
		\begin{align*}
			\Ebb{e^{-\theta M(\alpha,\delta)}}=e^{[\alpha (U(\delta))^{-1}]\ln{(1-\theta U_{\theta}(\delta))}}.
		\end{align*}
		We observe that since $\ln(1-x)+x\sim \frac{x^2}{2},$ as $x\to 0$,
		\[e^{[\alpha (U(\delta))^{-1}]\lbrb{\ln{(1-\theta U_{\theta}(\delta))}+\theta U_{\theta}(\delta)}}=e^{[\alpha (U(\delta))^{-1}]\frac{1}{2}\theta^2 U^2_{\theta}(\delta)\lbrb{1+o(1)}}.\]
		However, from \eqref{eq:Q-Potentials}, using $U_\theta(\delta)\leq U(\delta)$, we get that
		\begin{align*}
		&\lbbrbb{\frac{\alpha}{ U(\delta)}}U^2_{\theta}(\delta)\leq \alpha\frac{U^2_{\theta}(\delta)}{U(\delta)}+U^2_{\theta}(\delta)\leq \lbrb{1+\alpha}U(\delta)=o(1)
		\end{align*}
		since $U(0+)=0$ whenever $X$ is not a compound Poisson process. Therefore, we obtain that
		\[\Ebb{e^{-\theta M(\alpha,\delta)}}\sim e^{-[\alpha (U(\delta))^{-1}]\theta U_{\theta}(\delta))}.\]
		Next, we observe that since $U_\theta(\delta)\leq U(\delta)$
				\begin{equation}\label{eq:Utheta}
				\alpha\frac{U_{\theta}(\delta)}{U(\delta)}\leq \lbbrbb{\frac{\alpha}{ U(\delta)}}U_{\theta}(\delta)\leq \alpha.
				\end{equation}
		From \eqref{eq:Q-Potentials} we have that for some independent of $X$ exponential random variable $e_\theta$ with parameter\,$\theta>0$
		\begin{align*}
		&U(\delta)=\Ebb{\int_{0}^{\infty}1_{\{X_t\leq \delta\}}dt}=\Ebb{\int_{0}^{e_\theta}1_{\{X_t\leq \delta\}}dt}+\Ebb{\int_{e_\theta}^{\infty}1_{\{X_t\leq \delta\}}dt}=\\
		&\int_{0}^{\infty}\theta e^{-\theta v}\int_{0}^{v}\Pbb{X_t\leq \delta}dtdv+\Ebb{\int_{e_\theta}^{\infty}1_{\{X_t\leq \delta\}}dt1_{\{T_1(\delta)>e_\theta\}}}\leq \\
		&U_\theta(\delta)+U(\delta)\Pbb{T_1(\delta)>e_\theta}=U_\theta(\delta)+o(1)U(\delta).
		\end{align*}	
		Hence, from \eqref{eq:Utheta} we get that $\lim\tto{\delta}[\alpha (U(\delta))^{-1}] U_{\theta}(\delta)=\alpha$ and thus
		\begin{align*}
			\lim_{\delta\to 0} \Ebb{e^{-\theta M(\alpha,\delta)}}=e^{-\alpha \theta}.
		\end{align*}

		Hence, $M(\alpha,\delta)$ converges to $\alpha$ in probability. Clearly, for every $\gamma>0$, 
		\begin{align*}
			\Pbb{U(\delta)N(t,\delta)>\gamma}=\Pbb{\sum_{i=1}^{[\frac{\gamma} {U(\delta)}]}\eta_{i}\leq t}=\Pbb{M(\gamma,\delta)\leq t}.
		\end{align*}
		The convergence of $M(\gamma,\delta)$ then implies that
		\[\lim_{\delta\to 0}\Pbb{U(\delta)N(t,\delta)>\gamma}=1 \text{ if $\gamma<t$},\]
		\[\lim_{\delta\to 0}\Pbb{U(\delta)N(t,\delta)>\gamma}=0 \text{ if $\gamma>t$}.\]
	\end{proof}

Next, we deduce a representation of $N(1, \delta)$.
	 \begin{lemma}\label{L3}
		 With $N(1,\delta)$ specified as above, we have for all $j>0$,
		 \begin{equation}\label{eq3}
			 N\lbrb{1,\delta}\stackrel{d}{=} \sum_{i=1}^{j}N_{i}\lbrb{\frac{1}{j},\delta}+A_{j},
		 \end{equation}
		 where $N_{i}\lbrb{\frac{1}{j},\delta}$ are i.i.d.\,copies of $N\lbrb{\frac{1}{j},\delta}$ and $-j<A_{j}\leq 0$ is an integer valued random variable.
	 \end{lemma}

	 \begin{proof}
	 	The proof can be done easily by induction once one observes that, for any $0<t<1$, $N(1,\delta)\stackrel{d}{=}N(t,\delta)+N(1-t,\delta)+A$, where $A$ is either $0$ or $-1$ and $N(1-t,\delta)$ and $N(t,\delta)$ are independent. We argue pathwise. It is not difficult to see the following: Up to time $t$ we have 		some number of intervals that cover the range of the subordinator. At time $t$ we start a new covering of the remaining time $1-t$ and continue the old covering for the whole length $1$. It is easy to see that the new covering of the range on $(t,1)$ will exceed the old covering by at most $1$ 			since in the worst case scenario we have started our new covering at $X_{t}$ when $X_{t}$ was in an interval of the original covering of the whole $(0,1)$. 
	 \end{proof}
 \newpage
	 The next lemma is in fact \cite[Lem. 3.1]{SL} which we reproduce for clarity.
	\begin{lemma}\label{L4}
		There is an absolute constant $C_{a}>0$ such that for each $\delta>0$ and $x>0$
		\begin{equation}\label{eq4}
			\Pbb{N(t,\delta)\geq x}\leq e^{\frac{2C_{a}t}{U(\delta)}-\frac{x}{8}},
		\end{equation}
		where $C_a$ does not depend on $X$.
	 \end{lemma}
	 \begin{proof}
		By Markov's inequality we obtain for arbitrary $\lambda>0$ that
		 \begin{align*}
			\Pbb{N(t,\delta)\geq x}&=\Pbb{N(t,\delta)\geq [x]+1}\\
			 &=\Pbb{\sum_{i=1}^{[x]+1}\eta_{i}\leq t}\\
			 &\leq e^{\lambda t}\lbrb{\Ebb{e^{-\lambda T_{1}(\delta)})}}^{[x]+1}.
		 \end{align*}	
		 Using the estimates of \cite{SL} for $\lambda=2\Phi\lbrb{\frac{1}{\delta}}$ we obtain that
		 \begin{align*}
		 	\lbrb{\Ebb{e^{-\lambda T_{1}(\delta)}}}^{[x]+1}\leq e^{\frac{-[x]-1}{8}}
		\end{align*}
		 and therefore
		 \begin{align*}
			\Pbb{N(t,\delta)\geq x} &\leq e^{2\Phi\lbrb{\frac{1}{\delta}}t}e^{\frac{-[x]-1}{8}}.
		 \end{align*}
		 It remains to observe that according to \cite[Chap.3, Prop.1]{B96} there is an absolute constant $C_{a}$ not depending on the subordinator such that $\Phi\lbrb{\frac{1}{\delta}}\leq \frac{C_{a}}{U(\delta)}$. 
	 \end{proof}
 
	Our next aim is to estimate the variance of $N(t,\delta)$ via its second moment.
 	\begin{lemma}\label{L5}
		 For any subordinator $X$ we have that 
		 \begin{equation}\label{eq5}
			 \mathrm{Var}(N(t,\delta))\leq \Ebb{N^{2}(t,\delta)}\leq M\frac{ t^{2}}{U^{2}(\delta)}+K
		 \end{equation}
		 where $M$ and $K$ are absolute constants not depending on $X$.
	\end{lemma}
	\begin{proof}
		We divide 
		\begin{align*}
			\Ebb{N^{2}(t,\delta)}=\Ebb{N^{2}(t,\delta);N^{2}(t,\delta)]\leq b^{2}}+\Ebb{N^{2}(t,\delta);N^{2}(t,\delta)]> b^{2}}
		\end{align*}
		and we estimate both summands separately. For the rest of the proof we set $b=32C_{a}\frac{t}{U(\delta)}$ so that the first summand can be bounded by
		\begin{align*}
			\Ebb{N^{2}(t,\delta);N^{2}(t,\delta)\leq b^{2}}\leq\frac{1024C^{2}_{a}t^{2}}{U^{2}(\delta)}.
		\end{align*}
		For the second summand we obtain from integration by parts
		\begin{align*}
			\Ebb{N^{2}(t,\delta);N^{2}(t,\delta)> b^{2}}=b^{2}\Pbb{N^{2}(t,\delta)\geq b^{2}}+\int_{b^{2}}^{\infty}\Pbb{N(t,\delta)\geq \sqrt{y}}dy.
		\end{align*}
		The first summand of the right hand side is now estimated from above using Lemma \ref{L4} and recalling that $b=32C_{a}\frac{t}{U(\delta)}$ by
 		\begin{align*}
			b^{2}\Pbb{N^{2}(t,\delta)\geq b^{2}}\leq b^{2}e^{\frac{2C_{a}t}{U(\delta)}-\frac{32C_{a}t}{8U(\delta)}}\leq b^{2}.
		\end{align*}
 		It remains to estimate
		\begin{align*}
			\int_{b^{2}}^{\infty}\Pbb{N(t,\delta)\geq \sqrt{y}}dy=b^{2}\int_{1}^{\infty}\Pbb{N(t,\delta)\geq b\sqrt{z}}dz,
		\end{align*}
		which we do by using Lemma \ref{L4} in the following manner
		 \begin{align*}
			 b^{2}\int_{1}^{\infty}\Pbb{N(t,\delta)\geq b\sqrt{z}}dz&\leq b^{2}\int_{1}^{\infty}e^{\frac{2C_{a}t}{U(\delta)}-\frac{b\sqrt{z}}{8}}dz\\
			 &= b^{2}\int_{1}^{\infty}e^{-\frac{2C_{a}t}{U(\delta)}(2\sqrt{z}-1)}dz\\
			 &=2b^{2}\int_{1}^{\infty}ye^{-\frac{2C_{a}t}{U(\delta)}(2y-1)}dy\\
			&= b^{2}\int_{1}^{\infty}\frac{u+1}{2}e^{-\frac{2C_{a}t}{U(\delta)}u}du\\
			&\leq \frac{1024C^{2}_{a}t^{2}}{U^{2}(\delta)}\int_{1}^{\infty}ue^{-\frac{2C_{a}t}{U(\delta)}u}du\\
			&\leq 256\int_{0}^{\infty}ve^{-v}dv=:K.
		 \end{align*}
	 \end{proof}
	 
 \begin{lemma}\label{L2}
		 Let $X$ be a subordinator with $\Pi\lbrb{0,\infty}=\infty$ or, if $\Pi\lbrb{0,\infty}<\infty$, then $d>0$. Then there is a sequence $\delta_{j}\downarrow 0$ such that
		 \begin{align*}
		 	U(\delta_j)=r^j
		 \end{align*}
		 for some $r\in (0,1)$ 	and
		  \begin{equation}\label{eq2}
			 \lim_{j\to\infty} U(\delta_{j})N(1,\delta_{j})=1\quad a.s.  
		 \end{equation} 
	 \end{lemma} 
	 \begin{proof}
		 When either $\Pi\lbrb{0,\infty}=\infty$ or, if $\Pi\lbrb{0,\infty}<\infty$, then $d>0$, we have that $U(0+)=0$ since $\Pbb{X_t=0}=0,\forall t>0$. Since then $U(x)$ tends to $0$ as $x\to 0$ , we can always choose $\delta_{j}\downarrow 0$ such that $U(\delta_j)=r^j$. Assume without loss of generality that $t=1$. From \eqref{eq3} of Lemma \ref{L3} we get with  $N_i, 1\leq i\leq j^2$, independent copies of $N(\cdot,\cdot)$
		 \begin{align*}
			 Y_{i}(j^{2}):=U(\delta_{j})\lbrb{N_i\lbrb{\frac{1}{j^2},\delta_{j}}-\Ebb{N_i\lbrb{\frac{1}{j^2},\delta_{j}}}}
		\end{align*}
		that
		\begin{align*}
			U(\delta_{j})N(1,\delta_{j})-\Ebb{U(\delta_{j})N(1,\delta_{j})}=\sum_{i=1}^{j^{2}}Y_{i}(j^{2})+U(\delta_{j})\lbrb{A_{j^{2}}-\Ebb{A_{j^{2}}}}.
		\end{align*}
		From $|A_{j^{2}}|\leq j^{2}$ and the choice of $\delta_{j}$ we get that 
		\begin{align*}
			\lim_{j\to\infty}U(\delta_{j})(A_{j^{2}}-\Ebb{A_{j^{2}}})=\lim_{j\to\infty}r^{j}(A_{j^{2}}-\Ebb{A_{j^{2}}})=0\quad \text{a.s.}
		\end{align*}
		Then from Chebyshev's inequality, the i.i.d.\,property of $Y_{i}(j^{2})$ and \eqref{eq5} we get that for each $\epsilon>0$
		\begin{align*}
			 \Pbb{\labsrabs{\sum_{i=1}^{j^{2}}Y_{i}(j^{2})}>\epsilon}&\leq
			 \frac{\Ebb{\lbrb{\sum_{i=1}^{j^{2}}Y_{i}(j^{2})}^2}}{\epsilon^2}\\
			 &=\frac{U^{2}(\delta_{j})}{\epsilon^2}\Ebb{\sum_{i=1}^{j^{2}}\lbrb{N_i\lbrb{\frac{1}{j^2},\delta_{j}}-\Ebb{N_i\lbrb{\frac{1}{j^2},\delta_{j}}}}^2}  \\ &\leq\frac{1}{\epsilon^2}2j^{2}U^{2}(\delta_{j})\Ebb{N^{2}\lbrb{\frac{1}{j^2},\delta_{j}}}\leq \frac{1}{\epsilon^2}2Kj^{2}U^{2}(\delta_{j})+\frac{1}{\epsilon^2}\frac{M}{j^2}\\
			 &=\frac{1}{\epsilon^2}\lbrb{2Kr^{j}j^2+\frac{1}{j^2}}
		\end{align*}
		and the quantity at the right hand side is summable in $j$. Hence, the Borel-Cantelli lemma implies that
		\begin{align}\label{eq:toProve}
			\lim_{j\to\infty}\lbrb{ U(\delta_{j})N(1,\delta_{j})-\Ebb{U(\delta_{j})N(1,\delta_{j})}}=0\quad\text{a.s.}
		\end{align}

	Finally, we proceed to show that 
	\begin{align*}
		\lim_{\delta\to 0}U(\delta)\Ebb{N(1,\delta)}=1.
	\end{align*}
	For that purpose put $M(\delta)=U(\delta)N(1,\delta)$ and observe that from Lemma \ref{L1} and it follows that 
	\begin{align*}
		\Ebb{M(\delta), M(\delta)<1-\epsilon}\leq \lbrb{1-\epsilon}\Pbb{M(\delta)<1-\epsilon}\to 0,\,\,\text{ as $\delta\to 0$,}
	\end{align*}
	for any $\epsilon>0$. Moreover, with $A_{\delta,\epsilon}=1_{\{M(\delta)>1+\epsilon\}}$, Lemma \ref{L5} gives together with H\"{o}lder's inequality that 
	\begin{align*}
		\lbrb{\Ebb{M(\delta)1_{A_{\delta,\epsilon}}}}^{2}\leq \Pb^{2} \lbrb{A_{\delta,\epsilon}}\Ebb{M^{2}(\delta)}\leq \Pb^{2} \lbrb{A_{\delta,\epsilon}}(KU^2(\delta)+M).
	\end{align*}
	Therefore, for each $\epsilon>0$,
	\begin{align*}
		\,\,\,\,\,\,\,(1-\epsilon)\Pbb{M(\delta)\in(1-\epsilon,1+\epsilon)}
		&\leq \Ebb{M(\delta)}\\
		&\leq \Pb^{2} \lbrb{A_{\delta,\epsilon}}\lbrb{(KU^2(\delta)+M)}+(1+\epsilon)\Pbb{M(\delta)\leq 1+\epsilon}.
	\end{align*}
Therefore from Lemma \ref{L1}
\begin{align*}
&1-\epsilon\leq \liminf\tto{\delta}\Ebb{M(\delta)}\leq \limsup\tto{\delta}\Ebb{M(\delta)}\leq 1+\epsilon
\end{align*}
and sending $\epsilon\downarrow 0$ we get that $\lim\tto{\delta}\Ebb{M(\delta)}=1$ which then from \eqref{eq:toProve} yields the result.

	\end{proof}
	Now we prove Theorem 1.
	\begin{proof}[Proof of Theorem \ref{t1}]
		All we need to do is to extend Lemma \ref{L2} to arbitrary sequences. Suppose that 
		\[r^{i+1}=U\lbrb{\delta_{i+1}}\leq U\lbrb{\delta} \leq r^i=U\lbrb{\delta_i}.\] 
		Then by monotonicity of $U(\cdot)$ and $N(1,\cdot)$ we obtain that
		\begin{align*}
			rU(\delta_i)N(1,\delta_i)&=\frac{U(\delta_{i+1})}{U(\delta_i)}U(\delta_i)N(1,\delta_i)\\
			&\leq U(\delta)N(1,\delta)\\
			&\leq U(\delta_{i})N(1,\delta_{i+1})=\frac{1}{r}U(\delta_{i+1})N(1,\delta_{i+1}).
		\end{align*}
		As the left hand side converges to $r$ and the right hand side converges to $1/r$ as $\delta$ tends to zero, the proof is finished since $r$ is arbitrarily close to $1$.
	\end{proof}
	\section{Proofs for Applications}
	We discuss first Corollary \ref{cor:d>0}.
\begin{proof}[Proof of Corollary \ref{cor:d>0}]
    The proof follows from Theorem \ref{t1} and \cite[Prop.1]{DS11} by simple integration and putting $q=0$ since this result discusses the potential density $u(x)=\frac{dU(x)}{dx}$.
\end{proof} 
Next we consider Corollary \ref{cor:RV}
\begin{proof}[Proof of Corollary \ref{cor:RV}]
The result is immediate from Theorem \ref{t1} and the fact that
\[U(x)\stackrel{x\to 0}{\sim}\frac{1}{\Gamma(1+\alpha)\Phi\lbrb{\frac{1}{x}}}\stackrel{x\to 0}{\sim}\frac{x^\alpha}{\Gamma(1+\alpha) L\lbrb{\frac{1}{x}}},\]
see \cite[Ch III, p. 75]{B96}.
\end{proof}
\newpage
Next we prove Proposition \ref{prop:1}.
\begin{proof}[Proof of Proposition \ref{prop:1}] 
Let $\lbrb{\mathcal{A}_\delta}_{\delta<1}$ be a collection of coverings of $(X_s)_{s\leq 1}$ such that
\[\limsup\tto{\delta}\sum_{I_i\in \mathcal{A}_\delta}f(I_i)<\infty.\]

Then trivially, for any fixed $c\in(0,1)$, and $\delta>0$ small enough with $\mathcal{A}^c_\delta=\{I_i\in\mathcal{A}_\delta;|I_i|>c\delta\}$
\[\sum_{I_i\in \mathcal{A}_\delta}f(I_i)\geq |\mathcal{A}^c_\delta|f(c\delta)\geq D|\mathcal{A}^c_\delta|f(\delta),\]
where we have used that, for very small $\delta>0$, such that $\frac{3}{2}\ln\labsrabs{\ln \delta}\geq \ln\labsrabs{\ln c\delta}\geq \ln\labsrabs{\ln \delta}$, 
\[f(c\delta)=\frac{\ln\labsrabs{\ln c\delta}}{\Phi\lbrb{\frac{\ln\labsrabs{\ln c\delta}}{c\delta}}}\geq \frac{2c}{3}\frac{\ln\labsrabs{\ln \delta}}{\Phi\lbrb{\frac{\ln\labsrabs{\ln \delta}}{\delta}}}\geq Df(\delta) \]
since \cite[Chap 3., Prop. 1, (6)]{B96} together with the monotonicity of $\Phi$ give for a fixed $c\in\lbrb{0,1}$ that
\begin{align*}
&\Phi\lbrb{\frac{\ln\labsrabs{\ln c\delta}}{c\delta}}\leq\frac{3}{2c} \Phi\lbrb{\frac{2c}{3}\frac{\frac{3}{2}\ln\labsrabs{\ln \delta}}{c\delta}}=\frac{3}{2c}\Phi\lbrb{\frac{\ln\labsrabs{\ln \delta}}{\delta}}.
\end{align*}

 This lower bound together with \cite[Chap 3., Prop. 1]{B96}, the fact that \eqref{eq:notCauchy}  and Theorem \ref{t1} hold, gives an immediate contradiction upon assuming that 
\[0<\limsup\tto{\delta}\frac{|\mathcal{A}^c_\delta|}{N(1,\delta)}=\limsup\tto{\delta}U(\delta)|\mathcal{A}^c_\delta|\asymp\limsup\tto{\delta}\frac{|\mathcal{A}^c_\delta|}{\Phi\lbrb{\frac{1}{\delta}}}\]
since then we would have for some $K>0$
\begin{align*}
&\infty>\limsup\tto{\delta}\sum_{I_i\in \mathcal{A}_\delta}f(I_i)\geq K\limsup\tto{\delta}\Phi\lbrb{\frac{1}{\delta}}f(\delta)=K\limsup\ttinf{x}\frac{\Phi\lbrb{x}\ln\ln(x)}{\Phi\lbrb{x\ln\ln(x)}}=\infty.
\end{align*}
\end{proof}
%
%
%

\end{document}